\documentclass[preprint,12pt]{elsarticle}  

\usepackage{amssymb,latexsym}
\usepackage{amsmath,amsthm} 
\usepackage[colorlinks,citecolor={blue},linkcolor={blue}]{hyperref}

\theoremstyle{plain}
\newtheorem{theorem}{Theorem}[section]
\newtheorem{corollary}[theorem]{Corollary}
\newtheorem{lemma}[theorem]{Lemma}
\newtheorem{proposition}[theorem]{Proposition}

\theoremstyle{definition}

\theoremstyle{remark}

\numberwithin{equation}{section}
\newtheorem{example}[theorem]{Example}
\newtheorem{examples}[theorem]{Examples}

\newcommand{\minimum}{\min}
\newcommand{\italic}{\emph}

\newcommand{\gradient}{\nabla}
\renewcommand{\.}{\dot}
\renewcommand{\:}{\ddot}
\newcommand{\precedes}{\prec}
\newcommand{\before}{\unlhd}

\renewcommand{\iff}{\Leftrightarrow}
\renewcommand{\implies}{\Rightarrow}
\renewcommand{\impliedby}{\Leftarrow}
\newcommand{\hs}{\hspace{1mm}}

\DeclareSymbolFont{AMSb}{U}{msb}{m}{n}
\DeclareMathSymbol{\N}{\mathbin}{AMSb}{"4E}
\DeclareMathSymbol{\Z}{\mathbin}{AMSb}{"5A}
\DeclareMathSymbol{\R}{\mathbin}{AMSb}{"52}
\DeclareMathSymbol{\Q}{\mathbin}{AMSb}{"51}
\DeclareMathSymbol{\C}{\mathbin}{AMSb}{"43}

\journal{Discrete Mathematics}

\begin{document}

\begin{frontmatter}

\title{Reduced Criteria for Degree Sequences}
\author{Jeffrey W. Miller} 
\address{Division of Applied Mathematics\\
Brown University\\
182 George St.,
Providence, RI 02912}
\ead{jeffrey\_miller@brown.edu}
\ead[url]{http://www.dam.brown.edu/people/jmiller/}

\begin{abstract}
For many types of graphs, criteria have been discovered that give necessary and sufficient conditions for an integer sequence to be the degree sequence of such a graph.  These criteria tend to take the form of a set of inequalities, and in the case of the Erd\H{o}s-Gallai criterion (for simple undirected graphs) and the Gale-Ryser criterion (for bipartite graphs), it has been shown that the number of inequalities that must be checked can be reduced significantly. We show that similar reductions hold for the corresponding criteria for many other types of graphs, including bipartite $r$-multigraphs, bipartite graphs with structural edges, directed graphs, $r$-multigraphs, and tournaments. We also prove a reduction for imbalance sequences.
\end{abstract}


\date{}

\end{frontmatter}


\section{Introduction}
There is a family of results that give necessary and sufficient conditions for an integer sequence to be the degree sequence of a given type of graph. A well-known example is the Erd\H{o}s-Gallai criterion \cite{Erdos_1960}: given $d\in\Z^n$ such that $d_1\geq d_2\geq\cdots\geq d_n\geq 0$ and $\sum d_i$ is even, there exists a simple undirected graph on $n$ vertices with degrees $d_1,\dotsc,d_n$ if and only if 
$$\sum_{i = 1}^k d_i\leq k (k-1)+\sum_{i = k +1}^n \min\{k,d_i\}$$
for all $k\in\{1,\dotsc,n\}$.
It is natural to ask: Are these conditions minimal? For example, must one check the inequality for all $k\in\{1,\dotsc,n\}$?  It turns out that the answer is ``no''.  It was shown by Zverovich and Zverovich \cite{Zverovich_1992}, and later by Tripathi and Vijay \cite{Tripathi_2003}, that it is sufficient that the inequality hold for $k = m$ and for all $k<m$ such that $d_k>d_{k +1}$, where $m =\max\{i: d_i\geq i\}$. As a result, it is possible to reduce the number of inequalities to the cardinality of $\{d_1,\dotsc,d_n\}$ or less. This type of reduction enjoys both theoretical and practical utility. On the theoretical side, it facilitates further results relating to degree sequences. On the practical side, savings in computation can be realized for algorithms dealing with large graphs.  

In the same spirit as the Erd\H{o}s-Gallai result, a spectrum of degree sequence criteria have been discovered for diverse classes of graphs, including bipartite graphs \cite{Gale_1957,Ryser_1957}, bipartite $r$-multigraphs \cite{Berge_1958}, bipartite graphs with structural edges \cite{Anstee_1982}, directed graphs \cite{Fulkerson_1960,Chen_1966}, $r$-multigraphs \cite{Chen_1966,Chungphaisan_1974}, tournaments \cite{Landau_1953}, and imbalance sequences of directed graphs \cite{Mubayi_2001}. Given the benefits of the reduction described above, it would be desirable to obtain analogous reductions for these other criteria as well.

The purpose of this paper is to show that such reductions can indeed be obtained for all these types of graphs. Most of these results appear to be new. For those results that are old, we provide new proofs, since our approach provides clear intuitions about why they are true. Using basic notions from finite calculus and some observations about convex sequences, we find it possible to solve these problems in a unified way, leading to proofs that are exceedingly simple and highly interpretable.

In previous work, other authors have studied the problem of degree sequence criterion reduction in three of the classes of graphs addressed in this paper. 
\begin{itemize}
\item In the case of simple undirected graphs, several researchers \cite{Eggleton_1975,Li_1975,Zverovich_1992,Tripathi_2003,Dahl_2005,Barrus_2011} have noticed that the Erd\H{o}s-Gallai criterion can be reduced. The strongest general result among these is due to Zverovich and Zverovich \cite{Zverovich_1992}, who obtained a further improvement to the reduction described above.
Zverovich and Zverovich \cite{Zverovich_1992} also proved the interesting fact that all of the Erd\H{o}s-Gallai inequalities are satisfied if the length of the sequence exceeds a certain bound that depends only on the maximum and minimum degrees. Barrus, Hartke, Jao, and West \cite{Barrus_2011} recently extended this result to sequences with bounded gaps between degrees.
Dahl and Flatberg \cite{Dahl_2005} made the insightful observation that concave sequences play an essential role in reducing the Erd\H{o}s-Gallai criterion. It turns out that this observation can be vastly extended to cover many other classes of graphs as well (when appropriately generalized to ``almost concave'' sequences).
\item In the case of bipartite graphs, the celebrated Gale-Ryser theorem \cite{Gale_1957,Ryser_1957} provides a degree sequence criterion of a similar form. Zverovich and Zverovich \cite{Zverovich_1992} also address this case, proving a reduction similar to the one described above for simple undirected graphs. In the present work, we obtain a stronger reduction, via a much simpler proof.
\item In the case of tournaments, Landau \cite{Landau_1953} provided a degree sequence criterion that can also be reduced in a similar manner, as noted (without proof) by Beineke \cite{Beineke_1989}. We prove an even stronger reduction.
\end{itemize}

The paper is organized as follows. Section \ref{section:concavity} contains definitions and elementary results. In Subsection \ref{refresher}, we state some standard definitions and facts from finite calculus. In Subsection \ref{concave}, we introduce the essential notion of an almost concave sequence, and make a few elementary observations regarding concave and almost concave sequences. Subsection \ref{Fulkerson_Ryser} demonstrates the utility of almost concavity by giving a short new proof of a well-known theorem of Fulkerson and Ryser. The main results of the paper are in Section \ref{section:reductions}, where after some general remarks we prove a series of reduced degree sequence criteria, covering many classes of graphs. Section \ref{section:reductions} ends with a negative result: a ``counterexample'' illustrating that these reductions are nontrivial in the sense that they do not hold for every class of graphs with a degree sequence criterion of the type described above.


\section{Concave and almost concave sequences}
\label{section:concavity}

It turns out that the reductions to be proved in Section \ref{section:reductions} hinge upon certain properties that can be succinctly and intuitively described using finite calculus. All of the reductions can be proven without using finite calculus, but our experience is that it yields dividends, both conceptually and notationally.  

\subsection{A brief refresher on finite calculus}
\label{refresher}

In this subsection we state some standard definitions and results from finite calculus \cite{Graham_1994}. 
Given $a\in\R^n$, define $\gradient a\in\R^n$ by 
$$(\gradient a)_k = a_k - a_{k-1}$$
for $k\in\{1,\dotsc,n\}$, where by convention $a_0 = 0$. (The sequence $\gradient a$ is sometimes referred to as the ``backward difference'', as opposed to the ``forward difference'': $(\Delta a)_k = a_{k +1} - a_k$. For our purposes, it is more notationally convenient to work with backward differences, but of course all the results below could be restated in terms of forward differences.) We often find it preferable to use the following alternative notation: define 
\begin{align*}
& \.a =\gradient a, \,\,\,\mbox{and}\\
& \:a = \gradient (\gradient a).
\end{align*}
 

The following elementary identities may be easily verified.  
 
\begin{proposition}[Basic properties] Suppose $a,b\in\R^n$, $c\in\R$, and $m\in\Z$ with $m\geq 1$. Then
\label{basics}
\begin{enumerate}
\item $\gradient(a+b) =\gradient a +\gradient b$
\item $\gradient(c a) = c(\gradient a)$
\item $a_k -a_j=\sum_{i = j +1}^k\.a_i$ whenever $1\leq j<k\leq n$.
\end{enumerate}
\end{proposition}

\subsection{Concave and almost concave sequences}
\label{concave}

In this subsection, we introduce the essential notion of an ``almost concave'' sequence, and make a few elementary observations regarding concave and almost concave sequences. A sequence $a \in\R^n$ is said to be \emph{concave} if
$$\:a_k\leq 0 \mbox{ for all $k\,$ s.t. } 3 \leq k \leq n.$$
Note that by \ref{basics}(3) this is equivalent to
$$\.a_k-\.a_j\leq 0 \mbox{ for all $j,k\,$ s.t. } 2 \leq j<k\leq n.$$  
We say that $a \in\R^n$ is \emph{almost concave} if 
$$\.a_k-\.a_j\leq 1 \mbox{ for all $j,k\,$ s.t. } 2 \leq j<k\leq n.$$  
Similarly, $a\in\R^n$ is \emph{nonincreasing} if 
$$a_k-a_j\leq 0 \mbox{ for all $j,k\,$ s.t. } 1\leq j<k\leq n,$$  
and \emph{almost nonincreasing} if 
$$a_k-a_j\leq 1 \mbox{ for all $j,k\,$ s.t. } 1\leq j<k\leq n.$$  
Given  $1\leq j<k\leq n$, we say that $a \in\R^n$ satisfies one of the preceding properties \emph{ on} $(j,\dotsc,k)$ if $(a_j,\dotsc,a_k)$ satisfies that property --- for instance, $a$ is \emph{concave on $(j,\dotsc,k)$} if $(a_j,\dotsc,a_k)$ is concave. 

\begin{examples} Let $a,b\in\R^n$ and $c,d\in\R$.
\label{examples}
\begin{enumerate}
\item $a+b$ is concave if $a$ and $b$ are concave.
\item $a + b$ is almost concave if $a$ is concave and $b$ is almost concave.
\item $c a$ is concave if $a$ is concave and $c\geq 0$.
\item $a$ is concave if $a_k =\sum_{i\leq k} b_i$ for all $k\in\{1,\dotsc,n\}$ and $b$ is nonincreasing.
\item $a$ is almost concave if $a_k =\sum_{i\leq k} b_i$ for all $k\in\{1,\dotsc,n\}$ and $b$ is almost nonincreasing.
\item $a$ is concave if $a_k = c k+ d$ for all $k\in\{1,\dotsc,n\}$.  In this case, we say that $a$ is \emph{linear}. In particular, $a$ is linear on $(j,j +1,\dotsc,l)$ if $a_k =\sum_{i\leq k} b_i$ and $b_{j +1} =\cdots = b_l$.
\item $a$ is almost concave if there exists $m\in\Z$ such that $a_k =I(k\geq m)$ for all $k\in\{1,\dotsc,n\}$, (where $I(A)$ is $1$ if $A$ is true, and is $0$ otherwise).
\end{enumerate}
\end{examples}

The following lemma is the focal point of this section.  

\begin{lemma}[Almost concave sequences] Let $a\in\Z^n$. If $a_1\geq 0$, $a_n\geq 0$, and $a$ is almost concave, then $a_k\geq 0$ for all $k\in\{1,\dotsc,n\}$.
\label{concave_lemma}
\end{lemma}
\begin{proof}
Suppose not. Let $j =\min\{i: a_i<0\}$. Then $a_{j-1}\geq 0$ and $a_j<0$. Since $a$ is integer-valued, $\.a_j\leq-1$. Therefore, for any $k\in\{j+1,\dotsc,n\}$, we have $\.a_k =(\.a_k-\.a_j) +\.a_j\leq 0$ since $a$ is almost concave by assumption.  But then $a_n = a_j +\.a_{j+1} +\cdots +\.a_n<0$, a contradiction.
\end{proof}

It is easy to see that the lemma applies to concave sequences as well. The following special case is also useful.

\begin{corollary} \label{concave_corollary}
Let $a\in\Z^n$. If $a_1\geq 0$, $a_n\geq 0$, and there is some $l\in\{1,\dotsc,n\}$ such that $\:a_l\leq 1$ and $\:a_k\leq 0$ for all $k\in\{3,\dotsc,n\}$ except $l$, then $a_k\geq 0$ for all $k\in\{1,\dotsc,n\}$.
\end{corollary}
\begin{proof}
For any $j,k$ such that $1<j<k\leq n$, we have $\.a_k-\.a_j=\:a_{j+1} +\cdots +\:a_k\leq 1$ by our assumptions, so $a$ is almost concave.
\end{proof}

\subsection{An application to a theorem of Fulkerson and Ryser}
\label{Fulkerson_Ryser}

In this subsection, we demonstrate the utility of almost concavity by giving a short new proof of a theorem of Fulkerson and Ryser. Given $a\in\N^n$ (where $\N =\{0,1,2,\dotsc\}$), let $(a_{[1]},\dotsc,a_{[n]})$ denote a permutation of $a$ such that $a_{[1]}\geq\cdots\geq a_{[n]}$. Given $a,b\in\N^n$, we say (see \cite{Marshall_1979}) that $a$ \emph{is majorized by} $b$ and write $a\precedes b$ if $\sum_{i\leq k} a_{[i]}\leq\sum_{i\leq k} b_{[i]}$ for all $k\in\{1,\dotsc,n\}$ and $\sum_{i\leq n} a_i =\sum_{i\leq n} b_i$.  We write $a\leq b$ if $a_k\leq b_k$ for all $k$. Let $e^1 =(1,0,0,\dotsc,0)$, $e^2 =(0,1,0,\dotsc,0)$, and so on. The following theorem of Fulkerson and Ryser \cite{Fulkerson_1962} is not entirely trivial, since the modified vectors may no longer be nonincreasing. The results of Subsection \ref{concave} allow for a short proof.

\begin{theorem}[Fulkerson and Ryser] Suppose $a,b\in\N^n$ are nonincreasing and $1\leq j\leq l\leq n$ such that $a_j>0$ and $b_l>0$. If $a\precedes b$, then $a-e^j\precedes b-e^l$.
\end{theorem}
\begin{proof}
Suppose $a\precedes b$. Let $p =\max\{i: i\geq j, \, a_i= a_j\}$ and $q =\max\{i: i\geq l, \, b_i= b_l\}$. Note that $a-e^p$ is $a-e^j$ in nonincreasing order (and likewise for $b-e^q$ with respect to $b-e^l$), so $a-e^p\precedes b-e^q$ if and only if $a-e^j\precedes b-e^l$. For $k\in\{1,\dotsc,n\}$, let 
$$c_k = \sum_{i\leq k}(b_i-e_i^q)-\sum_{i\leq k}(a_i-e_i^p)=\sum_{i\leq k}b_i-\sum_{i\leq k}a_i-I(k\geq q) + I(k\geq p),$$
so $c\geq 0$ if and only if $a-e^p\precedes b-e^q$ if and only if $a-e^j\precedes b-e^l$. We show that $c\geq 0$. If $p\leq q$ then trivially $c\geq 0$, since $a\precedes b$. Suppose $q<p$. Note that $c_k\geq 0$ for all $k<q$ and all $k\geq p$. Thus $a_q =a_{q+1}=\cdots = a_p$, since $j\leq l\leq q<p$. Thus, on $(q-1,q,\dotsc,p)$, we have
$$c_k =\sum_{i\leq k}(b_i-e_i^q)-\sum_{i\leq k} a_i + I(k\geq p)=\mbox{(concave)+(linear)+(almost concave)},$$
by \ref{examples}(4), (6), and (7). Hence, $c$ is almost concave on $(q-1,q,\dotsc,p)$ by 2.2(1) and (2). Since $c_{q-1}\geq 0$ and $c_p\geq 0$, Lemma \ref{concave_lemma} implies that $c_k\geq 0$ for all $k\in\{q,\dotsc,p-1\}$.
\end{proof}

\section{Reduced criteria for degree sequences}
\label{section:reductions}

This section contains our main results --- a series of theorems providing reduced degree sequence criteria for many classes of graphs. We proceed in the following way.
Our first task is to introduce some order relations that play a key role.  From these relations arise the notion of generalized conjugates. Using generalized conjugates, we obtain a broad characterization of many of the results in this section.  With this in hand, we embark on a tour of degree sequence criteria for various classes of graphs, obtaining reductions for each class.  The section ends with a ``counterexample'' for which our primary reduction does not apply.

Throughout the paper, we use $\N$ to denote the nonnegative integers, $\{0,1,2,\dotsc\}$. Given $a,b\in\N^n$, we write
$$ a\leq b $$
if $a_k\leq b_k$ for all $k\in\{1,\dotsc,n\}$.  
Given $a\in\N^n$ and $b\in\N^m$, we write
$$ a\before b $$
if $\sum_{i\leq k} a_{i}\leq\sum_{i\leq k} b_{i}$ for all $k\in\{1,\dotsc,\max\{n,m\}\}$, with the convention that $a_i = 0$ for $i>n$ and $b_i = 0$ for $i>m$.  This relation differs from majorization (as in Subsection \ref{Fulkerson_Ryser}) in that we do not rearrange the elements to be nonincreasing, and we do not require the sums to be equal.  Next, we define a corresponding order relation on matrices. Let $\N^{m\times n}$ denote the set of $m\times n$ nonnegative integer matrices. Given $A,B\in\N^{m\times n}$, let us write
$$A\before B$$
if $a\before b$, where $a,b\in\N^n$ are the column sums of $A,B$, respectively.  Given a finite subset $S\subset\N^{m\times n}$ with a unique maximal element (with respect to $\before$), define the \emph{generalized conjugate of $S$} to be the column sums of the maximal element. To motivate the foregoing definition, consider:

\begin{example}[The conjugate as a generalized conjugate] \label{example:conjugate} Given $b\in\N^m$, the sequence $b'=(b'_1,b'_2,\dotsc)$ defined by $b'_k =\#\{i: b_i\geq k\}$ is called the \emph{conjugate} of $b$. (Note: We use $\# E$ to denote the number of elements in a set $E$.) If $n\geq\max b_i$ and $S\subset\{0,1\}^{m\times n}$ is the subset of binary matrices with row sums $b$, then $(b'_1,\dotsc,b'_n)$ coincides with the generalized conjugate of $S$.  When $b$ is nonincreasing, the maximal matrix is called the \emph{Ferrers diagram} of $b$.
\end{example}

The notion of a generalized conjugate seems to have developed gradually, making it difficult to credit a particular point of origin --- however, the idea is clearly present in the work of Anstee \cite{Anstee_1982} (on structured bipartite graphs), and appears fully formed in the work of Chen \cite{Chen_1992} (on structured bipartite multigraphs). Also, see \cite{Aigner_1984}.


For the generalized conjugates we deal with, explicit formulas can be given, and while these are useful for many purposes, the preceding abstract characterization often assists in understanding and proving some of their properties. For example, from the abstract characterization of $b'$ it is obvious that $\sum b_i =\sum b'_i$ and that if $b$ is nonincreasing then $b''= b$. These properties are also trivial to prove from the explicit formula, but they do not immediately suggest themselves, and this disparity becomes even more marked with more complicated general conjugates.

Now, before proceeding, let us pause to consider the general idea underlying several of the results of this section.  Many of the criteria for degree sequences are expressed as a sequence of inequalities that can be put in the form $a\before b$ with $a$ nonincreasing. (Take the Erd\H{o}s-Gallai criterion, for example.) As we will see, $a$ tends to be the degree sequence (or in bipartite cases, the degree sequence of one part), and $b$ tends to be a generalized conjugate. For $a\in\N^n$ nonincreasing, let us define the \emph{corners of $a$} to be the set of indices
$$ C(a) =\{k: a_k>a_{k +1}, \, 1\leq k\leq n\},$$
with the convention that $a_{n +1} = 0$ (so $n$ is included if $a_n>0$). Visualizing the Ferrers diagram, it is apparent that the set of corners coincides with the set $\{a'_1,a'_2,\dotsc\}-\{0\}$ where $a'$ is the standard conjugate (as in Example \ref{example:conjugate}).
The baseline reductions we obtain in Subsections \ref{section:bipartite-graphs}-\ref{section:tournaments} involve showing that it is sufficient to check the inequalities at the corners of $a$, and in most cases, further reductions are also obtained. Of course, this ``corner reduction'' does not hold in general for the relation $\before$ (for example, if $a = (2,1,1,1)$ and $b = (2,0,2,1)$, then $C(a) =\{1,4\}$ and we have $a_1\leq b_1$ and $\sum_{i = 1}^4 a_i\leq \sum_{i = 1}^4 b_i$, but $a\ntrianglelefteq b$). Still, given the generality with which it applies to degree sequence criteria, one might surmise that this reduction depends only on some fundamental property of graphs --- however, this is not so: in Subsection \ref{negative_example} we give a simple example illustrating that the corner reduction does not always apply to degree sequence criteria.

Nonetheless, the corner reduction holds in many cases, and we can capture the underlying reason for nearly all these cases via the following corollary to the lemma on almost concave sequences.  The geometric intuition is that $\sum_{i\leq k} a_i$ is a piecewise linear ``curve'' in $k$ (changing slope at the corners of $a$), and $\sum_{i\leq k} b_i$ is almost concave on each piece, so if the first ``curve'' lies below the second at the corners, then it lies below it everywhere between.

\begin{lemma} Let $a,b\in\N^n$ with $a$ nonincreasing and $b$ almost nonincreasing. If $\sum_{i\leq k} a_{i}\leq\sum_{i\leq k} b_{i}$ for all $k\in C(a)$, then $a\before b$.
\label{main_lemma}
\end{lemma}
\begin{proof}
Let $c_k =\sum_{i\leq k} b_i-\sum_{i\leq k} a_i$. With the conventions that $a_0 = 0$ and $a_{n +1} = 0$, suppose $0\leq j<l\leq n$ such that $a_j\neq a_{j+1} =\cdots = a_l\neq a_{l+1}$. Then on $(j,\dotsc,l)$,
$$c_k =\sum_{i\leq k} b_i-\sum_{i\leq k} a_i=\mbox{(almost concave) + (linear)}=\mbox{(almost concave)}.$$
By assumption, $c_j\geq 0$ and $c_l\geq 0$, hence $c_k\geq 0$ for all $k\in\{j,\dotsc,l\}$ by Lemma \ref{concave_lemma}.  This takes care of $c_k\geq 0$ for all $k\in\{1,\dotsc,m\}$ where $m =\max C(a) =\max\{i: a_i>0\}$. Since by assumption $c_m\geq 0$, we have $0\leq c_m\leq c_{m+1}\leq\cdots\leq c_n$ (because $a_{m +1} =\cdots = a_n = 0$).
\end{proof}

So in each case, the corner reduction argument boils down to showing that the generalized conjugate under consideration is nonincreasing or almost nonincreasing.  Such a reduction is perhaps clear when the generalized conjugate is nonincreasing, but it is more subtle in the almost nonincreasing cases, and this may explain why those reductions have not been previously discovered. 

An interesting consequence of such a reduction is that many times it is possible to reformulate a criterion directly in terms of the standard conjugate (defined in Example \ref{example:conjugate}). In addition to facilitating further theoretical results involving degree sequences, this also has practical utility, since when dealing with large graphs with small degrees, it is desirable to represent degree sequences in a more compact form such as the conjugate or the sequence of counts $\#\{i: a_i = k\}$. Toward this end, we make note of the following relationships between $a$ and $a'$.
Let $a\in\N^n$ be nonincreasing. By visualizing the corners in the Ferrers diagram, we see that
\begin{align*}
a'_{a_k} = k &\,\mbox{ for all }\, k\in C(a),\\
a_{a'_j} = j &\,\mbox{ for all }\, j\in C(a'),
\end{align*}
and
\begin{align} \label{equation:coordinates}
\{(k,a_k): k\in C(a)\} =\{(a'_j,j): j\in C(a')\}.
\end{align}
The following observation is used several times in what follows.

\begin{proposition} \label{conjugate}
Let $a,b\in\N^n$ with $a$ nonincreasing. Let $l\in\{1,\dotsc,n\}$. The following are equivalent:
\begin{enumerate}
\item $\sum_{i\leq k} a_i\leq\sum_{i\leq k} b_i$ for all $k\in C(a)$ such that $k\leq l$
\item $j a'_j+\sum_{i>j} a'_i\leq\sum_{i\leq a'_j} b_i$ for all $j \in C(a')$ such that $a'_j\leq l$.
\end{enumerate}
\end{proposition}
\begin{proof}
It follows from Equation \ref{equation:coordinates} that $\{(k,a_k): k\in C(a), \, k\leq l\} =\{(a'_j,j): j\in C(a'), \, a'_j\leq l\}$. Also, it is apparent from the Ferrers diagram that for any such pair $(k,j)$ (that is, $k\in C(a)$ and $j = a_k$), we have that $\sum_{i\leq k} a_i = j a'_j +\sum_{i>j} a'_i$ and (since $k = a'_j$) that $\sum_{i\leq k} b_i =\sum_{i\leq a'_j} b_i$.
\end{proof}

One last point before we embark on our tour of degree sequence criteria reductions --- to fix terminology and notation: \emph{simple} graphs have no multiple edges and no loops, \emph{multigraphs} may have multiple edges but no loops, and all graphs are undirected unless specified otherwise. We use $x\wedge y$ to denote $\min\{x,y\}$, and $x\vee y$ to denote $\max\{x,y\}$.

\subsection{Bipartite graphs}
\label{section:bipartite-graphs}

\subsubsection{Gale-Ryser criterion for bipartite graphs}

Although we will obtain it shortly as a consequence of the reduction of Berge's criterion (as well as Anstee's criterion), it is instructive to first treat the Gale-Ryser criterion separately, since in some sense it is the simplest case.  Given $a\in\N^n$ and $b\in\N^m$, let us say that $(a,b)$ is \emph{bigraphic} if there is a bipartite graph on $n+m$ vertices with degree sequences $a, b$ in each part respectively. Given $a\in\N^n$ and $b\in\N^m$ with $a$ nonincreasing and $\sum a_i =\sum b_i$, Gale \cite{Gale_1957} and Ryser \cite{Ryser_1957} proved that $(a,b)$ is bigraphic if and only if 
\begin{equation}\tag{GR$_k$}
\sum_{i\leq k} a_i\leq\sum_{i\leq k} b'_i
\end{equation}
for all $k\in\{1,\dotsc,n\}$, that is, if and only if $a\before b'$. Thus, in this case, the appropriate generalized conjugate is simply the standard conjugate (as in Example \ref{example:conjugate}). (See \cite{Krause_1996} for a splendid proof of the Gale-Ryser theorem due to Krause.) From this we obtain:

\begin{theorem}[Reduced Gale-Ryser]\label{Gale_Ryser} Let $a\in\N^n$ and $b\in\N^m$ with $a$ nonincreasing and $\sum_{i\leq n} a_i =\sum_{i\leq m} b_i$. The following are equivalent:
\begin{enumerate}
\item $(a,b)$ is bigraphic
\item $a\before b'$
\item \textup{(GR$_k$)} for all $k\in\{1,\dotsc,n\}$
\item \textup{(GR$_k$)} for all $k\in C(a)$ such that $k<\max b_i$
\item $j a'_j+\sum_{i>j} a'_i\leq\sum_{i\leq a'_j} b'_i$ for all $j \in C(a')$ such that $a'_j<\max b_i$.
\end{enumerate}
\end{theorem}
\begin{proof} 
(1)$\iff$(3): Gale-Ryser. (2)$\iff$(3): Trivial. (3)$\implies$(4): Trivial. (3)$\impliedby$(4): Let $c_k =\sum_{i\leq k}(b'_i-a_i)$. Let $l =\max b_i$, noting that $\max b_i\leq n$ (since otherwise (GR$_{\max C(a)}$) would be violated). For $k\in\{l+1,\dotsc,n\}$ we have $\.c_k =-a_k\leq 0$ since $b'_k = 0$. Hence, $c_l\geq\cdots\geq c_n = 0$. To show that $c_1,\dotsc,c_{l-1}\geq 0$, the idea is to use the fact that $c$ is concave wherever $a$ is constant. Formally: since $b'$ is nonincreasing, we can apply Lemma \ref{main_lemma} to $(a_1,\dotsc,a_l)$ and $(b'_1,\dotsc,b'_l)$ to see that $c_1,\dotsc,c_l\geq 0$.
(4)$\iff$(5): Proposition \ref{conjugate}.
\end{proof} 

Using a set of inequalities that are easily shown to be equivalent to (GR$_k$), Zverovich and Zverovich \cite{Zverovich_1992} gave a rather intricate proof that it suffices to check the inequalities for $k\in C(a)$ --- however, this is a slightly weaker result than (1)$\iff$(4), and our proof is considerably simpler.

It is worth pointing out that the set of corners that must be checked can sometimes be reduced even further, by using the following fact: If $k_1<k_2<k_3$ are three consecutive corners such that $a_{k_1} = a_{k_2}+1 = a_{k_3}+2$ and both (GR$_{k_1}$) and (GR$_{k_3}$) hold, then (GR$_{k_2}$) holds as well. This can be proved using Lemma \ref{concave_lemma} and the fact that $c$ is almost concave on $(k_1,\dotsc,k_3)$.

In fact, it would seem that one could continue reducing the set of indices that need to be checked, at the expense of increasingly complicated descriptions of this set.
What would be more interesting, instead, would be salient special cases for which significant further reductions can be obtained. Here are some examples of the latter.

It is helpful to introduce the following terminology: given $a\in\N^n$ and $t\in\N$ such that $t\geq 1$, let us say that
\begin{itemize}
\item $a$ is \emph{$t$-dense} if every interval $(k,k +1,\dotsc,k+t-1)$ of length $t$ contained in $(\min a_i,\dotsc,\max a_i)$ contains an element of $a$ (that is, for each $k\in\{\minimum a_i,\dotsc,\max a_i-t +1\}$ there exists $i$ such that $k\leq a_i\leq k+t-1$)
\item $a$ is \emph{$t$-deep} if $a$ contains $t$ or more copies of each element of $\{\min a_i,\dotsc,$ $\max a_i-1\}$ (that is, for each $k\in\{\minimum a_i,\dotsc,\max a_i-1\}$ we have $\#\{i: a_i = k\}\geq t$).
\end{itemize}
(Some other authors \cite{Barrus_2011} have used the term ``gap-free'' instead of ``$1$-dense''.) 
When $a\in\N^n$ is nonincreasing and $t\in\N$ such that $t\geq 1$, by visualizing the Ferrers diagram it is clear that
\begin{itemize}
\item $a$ is $t$-dense if and only if $\. a_k\geq-t$ for all $k\in\{1,\dotsc,n\}$
\item $a$ is $t$-deep if and only if $\. a'_k\leq-t$ for all $k\in\{\minimum a_i +1,\dotsc,\max a_i\}$ (where $\. a'=\gradient(a')$), assuming that $\minimum a_i\geq 1$
\item $a$ is $1$-dense if and only if $a$ is $1$-deep.
\end{itemize}
The proof of the following is a nice example of the clarity afforded by the finite calculus perspective.

\begin{proposition} 
Let $a\in\N^n$ and $b\in\N^m$ with $\sum_{i\leq n} a_i =\sum_{i\leq m} b_i$, $1\leq a\leq m$, and $1\leq b\leq n$. Then $(a,b)$ is bigraphic if any of the following hold:
\begin{enumerate}
\item $a$ is $t$-dense and $b$ is $t$-deep, for some $t\in\N$, $t\geq 1$.
\item $a$ and $b$ are $1$-dense.
\item $a =(p,\dotsc,p)$ and $b =(q,\dotsc,q)$ for some $p,q\in\N$.
\end{enumerate}
\end{proposition}
\begin{proof}
(1)
We may assume $a$ to be nonincreasing, since it affects neither the assumptions nor the conclusion.
Let $c_k =\sum_{i\leq k} (b'_i-a_i)$ for $k\in\{1,\dotsc,n\}$. Then $\.c_k = b'_k-a_k$, and
\begin{enumerate}
\item[(i)] $\.c_k\geq 0$ for $1\leq k\leq\minimum b_i$ (since $k\leq\minimum b_i$ implies $b'_k = m\geq a_k$)
\item[(ii)] $\.c_k\leq 0$ for $\max b_i<k\leq n$ (since $k>\max b_i$ implies $b'_k = 0$)
\item[(iii)] $\:c_k=\.b'_k-\.a_k\leq-t + t = 0$ for $\minimum b_i<k\leq \max b_i$ (since $a$ is $t$-dense and $b$ is $t$-deep).
\end{enumerate}
By (i), $0\leq c_1\leq\cdots\leq c_{\minimum b_i}$. By (ii), $c_{\max b_i}\geq\cdots\geq c_n=0$, where $c_n = 0$ since $\sum_{i\leq n} a_i =\sum_{i\leq m} b_i$ and $\max b_i\leq n$ by assumption. Thus, by (iii) and Lemma \ref{concave_lemma}, $c_k\geq 0$ for $\minimum b_i<k<\max b_i$. Hence, $c\geq 0$, so $a\before b'$, and we can apply the Gale-Ryser theorem.

(2) Apply (1) and the fact that $1$-dense implies $1$-deep.

(3) Apply (2).
\end{proof}

It is also amusing to note that $b'_{\max b_i}\geq \max a_i$ is a sufficient condition (when $a\in\N^n$ and $b\in\N^m$ with $\sum a_i =\sum b_i$).

\subsubsection{Berge's criterion for bipartite $r$-multigraphs}

Next, we generalize from bipartite graphs to bipartite multigraphs.
By a \emph{bipartite $r$-multigraph} we mean a bipartite graph allowing multiple edges but with no more than $r$ edges connecting each pair of vertices. Let $a\in\N^n$ and $b\in\N^m$ with $a$ nonincreasing and $\sum a_i =\sum b_i$. Let $r\in\N$ with $r\geq 1$, and let $S\subset\N^{m\times n}$ be the subset of matrices with row sums $b$ and all entries less or equal to $r$. Denoting the corresponding generalized conjugate by $b^B$, it is easy to verify that $b^B$ exists as long as $\max b_i\leq r n$, that $b^B$ is nonincreasing, and that
$$\sum_{i\leq k} b^B_i =\sum_{i\leq m} (r k)\wedge b_i=\sum_{i\leq r k} b'_i. $$
Berge \cite{Berge_1958} proved the following criterion in the case of $m = n$, and it holds also for $m\neq n$ (as noted in \cite{Marshall_1979}, p.185): there exists a bipartite $r$-multigraph with degrees $(a,b)$ if and only if 
\begin{equation}\tag{B$_k$}
\sum_{i\leq k} a_i\leq\sum_{i\leq m} (r k)\wedge b_i
\end{equation}
for all $k\in\{1,\dotsc,n\}$. This is easily proved by a direct generalization of Krause's proof \cite{Krause_1996} of the $r = 1$ case (that is, the Gale-Ryser theorem).
We obtain:

\begin{theorem}[Reduced Berge] Let $a\in\N^n$ and $b\in\N^m$ with $a$ nonincreasing, $\max b_i\leq rn$, and $\sum a_i =\sum b_i$. Let $r\in\N$ with $r\geq 1$. The following are equivalent:
\begin{enumerate}
\item there exists a bipartite $r$-multigraph with degrees $(a,b)$
\item $a\before b^B$
\item \textup{(B$_k$)} for all $k\in\{1,\dotsc,n\}$
\item \textup{(B$_k$)} for all $k\in C(a)$ such that $k<\max b_i/r$
\item $j a'_j+\sum_{i>j} a'_i\leq\sum_{i\leq r a'_j} b'_i$ for all $j \in C(a')$ such that $a'_j<\max b_i/r$.
\end{enumerate}
\end{theorem}
\begin{proof} 
(1)$\iff$(3): Berge. (2)$\iff$(3): Trivial. (3)$\implies$(4): Trivial. (3)$\impliedby$(4): Since $b^B_k = 0$ for all $k\geq 1+\max b_i/r$, it follows that (B$_k$) holds for all $k\geq\max b_i/r$ (as in the proof of Theorem \ref{Gale_Ryser}).  Apply Lemma \ref{main_lemma} using that $b^B$ is nonincreasing. 
(4)$\iff$(5):  Proposition \ref{conjugate}.
\end{proof} 

As before, the set of corners that must be checked can be reduced even further in certain cases.

\subsection{Directed graphs and structured bipartite graphs}

The results of this subsection are more interesting, since they necessitate the use of almost concavity, as opposed to the more straightforward concave examples so far.

\subsubsection{Anstee's criterion for structured bipartite graphs}

By structured bipartite graphs, we mean bipartite graphs with certain edges required to be present or absent, in a way that will be made precise below. It is more clear to formulate this example in terms of matrices (rather than graphs), with the obvious corollaries for graphs. Fix $b\in\N^m$ and $C\in\{0,1\}^{m\times n}$. We say that $C$ is \emph{$b$-fillable} if its row sums do not exceed $b$ and it has no more than one nonzero entry per column --- that is,
$$ \sum_j C_{ij}\leq b_i \mbox{ for all $i$, and }\sum_i C_{ij}\leq 1 \mbox{ for all $j$}.$$
Given $A\in\{0,1\}^{m\times n}$, we say that \emph{$A$ fills $C$} if $A_{i j} = 1$ whenever $C_{i j} = 1$. (Sometimes, the set of such matrices is said to have ``structural ones'' in these entries.)
Given $C$ $b$-fillable, let $S_{1,C}\subset\{0,1\}^{m\times n}$ be the subset of binary matrices $A$ having row sums $b$ and filling $C$. Define $b^{1,C}$ to be the corresponding generalized conjugate, assuming $\max b_i\leq n$. (It should be clear that $S_{1,C}$ is nonempty and has a unique maximal element as long as $\max b_i\leq n$.) For our purposes it is preferable to work with the abstract definition, but it can be verified that
$$\sum_{i\leq k} b^{1,C}_i =\sum_{i\leq m} k\wedge(b_i-C_{i,k+1}-\cdots-C_{i n}). $$

Similar definitions can be made for ``structural zeros''. We say that $C$ is \emph{$b$-avoidable} if its row sums do not exceed $n-b$ and it has no more than one nonzero entry per column --- that is,
$$\sum_j C_{ij}\leq (n-b_i) \mbox{ for all $i$, and } \sum_i C_{ij}\leq 1 \mbox{ for all $j$}.$$
Given $A\in\{0,1\}^{m\times n}$, we say that \emph{$A$ avoids $C$} if $A_{i j} = 0$ whenever $C_{i j} = 1$. 
Given $C$ $b$-avoidable, let $S_{0,C}\subset\{0,1\}^{m\times n}$ be the subset of binary matrices $A$ having row sums $b$ and avoiding $C$. Define $b^{0,C}$ to be the corresponding generalized conjugate.  It can be verified that
$$\sum_{i\leq k} b^{0,C}_i =\sum_{i\leq m} (k-C_{i1}-\cdots-C_{ik})\wedge b_i. $$

Anstee \cite{Anstee_1982} has given a criterion that appears as the implications (1)$\iff$(2) in the following theorem. Anstee's results can be easily proven with a minor modification of Krause's proof \cite{Krause_1996} of Gale-Ryser. For $\beta\in\{0,1\}$ and $a\in\N^n$, consider the inequalities:
\begin{equation}\tag{A($\beta$)$_k$}
\sum_{i\leq k} a_i\leq\sum_{i\leq k} b^{\beta,C}_i.
\end{equation}

\begin{theorem}[Reduced Anstee] \label{Anstee}
Let $a\in\N^n$ and $b\in\N^m$ with $a$ nonincreasing, $\max b_i\leq n$, and $\sum a_i =\sum b_i$. Let $C\in\{0,1\}^{m\times n}$. 
If $C$ is $b$-fillable, then the following are equivalent:
\begin{enumerate}
\item there exists a binary matrix that fills $C$ and has column and row sums $a,b$ respectively
\item $a\before b^{1,C}$
\item \textup{(A(1)$_k$)} for all $k\in\{1,\dotsc,n\}$
\item \textup{(A(1)$_k$)} for all $k\in C(a)$
\end{enumerate}
If $C$ is $b$-avoidable, then the same results hold, replacing ``fills'' with ``avoids'', $b^{1,C}$ with $b^{0,C}$, and $A(1)_k$ with $A(0)_k$.
\end{theorem}
\begin{proof} 
(1)$\iff$(3): Anstee. (2)$\iff$(3): Trivial. (3)$\implies$(4): Trivial. (3)$\impliedby$(4): The idea of the proof is that the sequence $c_k =\sum_{i\leq k} (b_i^{1,C}-a_i)$ is almost concave wherever $a$ is constant. More precisely, we will apply Lemma \ref{main_lemma} by showing that $b^{1,C}$ is almost nonincreasing, and this will come from the fact that we have restricted $C$ to have most one nonzero entry per column.  Let $A$ be the maximal matrix in $S_{1,C}$. Then $A$ has column sums $b^{1,C}$. Let $j,k$ such that $1\leq j<k\leq n$. For any $i\in\{1,\dotsc,m\}$, if $A_{i j} = 0$ then $A_{i k} = 0$ unless $C_{i k} = 1$ (for otherwise, the matrix obtained by setting $A_{i j} = 1$ and $A_{i k} = 0$ would be strictly greater, and still belong to $S_{1,C}$.) Since $C$ has at most one nonzero entry per column, $A_{i j} \geq A_{i k}$ for all $i\in\{1,\dotsc,m\}$ except possibly one such $i$. Hence, $b^{1,C}_k-b^{1,C}_j\leq 1$, and therefore $b^{1,C}$ is almost nonincreasing.

The proof for $b^{0,C}$ is nearly identical, with the obvious changes, along with the observation that (when $1\leq j<k\leq n$ and $A$ is the maximal matrix in $S_{0,C}$), if $A_{i k} = 1$ then $A_{i j} = 1$ unless $C_{i j} = 1$.
\end{proof}

\subsubsection{Fulkerson's criterion for directed graphs}

We can apply the rather general criterion of Anstee to a case of particular interest: simple directed graphs. Given $a,b\in\N^n$, we say that $(a,b)$ is \emph{digraphic} if there exists a simple directed graph on $n$ vertices having out-degrees $a$ and in-degrees $b$ (that is, vertex $i$ has $a_i$ outgoing edges and $b_i$ incoming edges).  Such a graph can be represented by a $n\times n$ binary matrix with zeros on the diagonal, and having column and row sums $a,b$ respectively. Since there is only a single structural zero in each column, we can choose $C$ to be the $n\times n$ identity matrix $I$, and apply Anstee's criterion for matrices that avoid $I$. Given $a,b\in\N^n$, \italic{both} nonincreasing, with $\sum a_i =\sum b_i$, Fulkerson \cite{Fulkerson_1960} proved that $(a,b)$ is digraphic if and only if
\begin{equation}\tag{F$_k$}
\sum_{i\leq k} a_i\leq \sum_{i\leq k} (k-1)\wedge b_i +\sum_{i>k} k\wedge b_i
\end{equation}
for all $k\in\{1,\dotsc,n\}$. Anstee's criterion generalizes this to allow for arbitrary $a,b$ since we can always permute them together to make $a$ nonincreasing. (This generalization is due to Chen \cite{Chen_1966}.)
Note that 
$$\sum_{i\leq k} b^{0,I}_i=\sum_{i\leq k} (k-1)\wedge b_i +\sum_{i>k} k\wedge b_i =\sum_{i\leq k} b'_i-\#\{i: 1\leq i\leq k, \hs b_i\geq k\},$$
and when $b$ is nonincreasing we have $\#\{i: 1\leq i\leq k, \hs b_i\geq k\} = k\wedge b'_k$. The following is nearly a direct consequence of Theorem \ref{Anstee}.

\begin{theorem}[Reduced Fulkerson] Let $a,b\in\N^n$ with $a$ nonincreasing, $\max b_i\leq n-1$, and $\sum a_i =\sum b_i$. The following are equivalent:
\begin{enumerate}
\item $(a,b)$ is digraphic
\item $a\before b^{0,I}$
\item \textup{(F$_k$)} for all $k\in\{1,\dotsc,n\}$
\item \textup{(F$_k$)} for all $k\in C(a)$ such that $k\leq\max b_i$.
\end{enumerate}
If, further, $b$ is nonincreasing, then these are equivalent to
\begin{enumerate}
\item[(5)] $j a'_j+\sum_{i>j} a'_i\leq (\sum_{i\leq a'_j} b'_i)-a'_j\wedge b'_{a'_j}$ for all $j \in C(a')$ such that $a'_j\leq\max b_i$.
\end{enumerate}
\end{theorem}
\begin{proof}
(1)$\iff$(2): Theorem \ref{Anstee}. (2)$\iff$(3): Trivial. 
(3)$\implies$(4): Trivial. (3)$\impliedby$(4): By Theorem \ref{Anstee}, it suffices for (F$_k$) to hold for all $k\in C(a)$. Further, (F$_k$) holds for all $k>\max b_i$, since $b^{0,I}_k =0$ for all $k>1+\max b_i$ (and $\sum a_i =\sum b_i$).  (4)$\iff$(5): Proposition \ref{conjugate}.
\end{proof}


\subsubsection{Mubayi-Will-West criterion for imbalance sequences}

Given $d\in\Z^n$, we say that $d$ is an \emph{imbalance sequence} if there exists a simple directed graph with out-degrees $a$ and in-degrees $b$ such that $d_i = a_i-b_i$.  Strictly speaking, an imbalance sequence is not a degree sequence, however, Mubayi, Will, and West \cite{Mubayi_2001} give a criterion for imbalance sequences that we can treat similarly to the degree sequence criteria.

\begin{theorem}[Reduced Mubayi-Will-West] Let $d\in\Z^n$ be nonincreasing with $\sum d_i = 0$. The following are equivalent:
\begin{enumerate}
\item $d$ is an imbalance sequence
\item $\sum_{i\leq k} d_i\leq k(n-k)$ for all $k\in\{1,\dotsc,n\}$
\item $\sum_{i\leq k} d_i\leq k(n-k)$ for all $k\in\{1,\dotsc,n-1\}$ such that $d_k-d_{k+1}\geq 3$.
\end{enumerate}
\end{theorem}
\begin{proof} 
(1)$\iff$(2): Mubayi, Will, and West. (2)$\implies$(3): Trivial. (2)$\impliedby$(3): If $c_k = k(n-k)-\sum_{i\leq k} d_i$ then for $k\in\{2,\dotsc,n\}$ we have $\:c_k =-2+ d_{k-1}-d_k$, which is less or equal to $0$ as long as $d_{k-1}-d_k\leq 2$. Thus, $c$ is concave on any interval $(j,\dotsc,l)$ such that $d_{k-1}-d_k\leq 2$ for all $k\in\{j+2,\dotsc,l\}$.  Apply Lemma \ref{concave_lemma}.
\end{proof} 

This implies some interesting sufficient (but not necessary) conditions.

\begin{corollary} 
\label{imbalance_corollary} Let $d\in\Z^n$ be nonincreasing with $\sum d_i = 0$. Then $d$ is an imbalance sequence if $d_k-d_{k+1}\leq 2$ for all $k\in\{1,\dotsc,n-1\}$. In particular, it is an imbalance sequence if $\{d_1,\dotsc,d_n\}$ contains every even integer or every odd integer (or every integer) between $\min d_i$ and $\max d_i$.
\end{corollary}

\subsection{Undirected graphs}

\subsubsection{Chen's criterion for $r$-multigraphs}

An \emph{$r$-multigraph} is a (loopless) graph allowing multiple edges, but with no more than $r$ edges connecting each pair of vertices. An $r$-multigraph with degree sequence $a$ can be represented by a symmetric nonnegative integer matrix with zero diagonal, column and row sums $a$, and all entries less or equal to $r$.  Given $a\in\N^n$ and $r\in\N$, we say that $a$ is \emph{$r$-multigraphic} if there exists an $r$-multigraph with degrees $a$. Given $a\in\N^n$ nonincreasing such that $\sum a_i$ is even, Chungphaisan \cite{Chungphaisan_1974} has shown that $a$ is $r$-multigraphic if and only if
\begin{equation}\tag{C1$_k$}
\sum_{i\leq k} a_i\leq r k(k-1)+\sum_{i>k}(r k)\wedge a_i
\end{equation}
for all $k\in\{1,\dotsc,n\}$. Curiously, this criterion does not lend itself to description in terms of generalized conjugates --- however, there is an equivalent criterion that does. Let $S\subset\N^{n\times n}$ be the subset of matrices with zero diagonal, row sums $a$, and all entries less or equal to $r$. Let $a^C$ be the corresponding generalized conjugate, assuming $\max a_i\leq r(n-1)$. One can check that
$$\sum_{i\leq k} a^C_i =\sum_{i\leq k} (r(k-1))\wedge a_i +\sum_{i>k} (r k)\wedge a_i,$$
and when $a$ is nonincreasing we have that if $a_k\geq r k$ then $\sum_{i\leq k} a^C_i =\sum_{i\leq r k} (a'_i-1)$.
Given $a\in\N^n$ nonincreasing such that $\sum a_i$ is even, Chen \cite{Chen_1966} proved that $a$ is $r$-multigraphical if and only if
\begin{equation}\tag{C2$_k$}
\sum_{i\leq k} a_i\leq \sum_{i\leq k} (r(k-1))\wedge a_i +\sum_{i>k} (r k)\wedge a_i
\end{equation}
for all $k\in\{1,\dotsc,n\}$. Note that (C1$_k$) and (C2$_k$) are identical when $a_k\geq r(k-1)$ (and $a$ is nonincreasing).

\begin{theorem}[Reduced Chen] \label{Chen} Let $a\in\N^n$ and $r\in\N$ such that $a$ is nonincreasing, $\max a_i\leq r(n-1)$, and $\sum a_i$ is even. Let $m =\max\{i: a_i\geq r(i-1)+1\}$. The following are equivalent:
\begin{enumerate}
\item $a$ is $r$-multigraphic
\item $a\before a^C$
\item \textup{(C1$_k$)} for all $k\in\{1,\dotsc,n\}$
\item \textup{(C1$_k$)} for $k = m$ and all $k\in C(a)$ such that $k<m$
\item \textup{(C2$_k$)} for $k = m$ and all $k\in C(a)$ such that $k<m$
\item \textup{(C2$_m$)} and $j a'_j+\sum_{i>j} a'_i\leq\sum_{i\leq r a'_j} (a'_i-1)$ for all $j \in C(a')$ such that $a'_j<m$. 
\end{enumerate}
\end{theorem}
\begin{proof} 
(1)$\iff$(2): Chen. 
(1)$\iff$(3): Chungphaisan. 
(4)$\iff$(5): (C1$_k$) and (C2$_k$) are identical when $k\leq m$.
(3)$\implies$(4): Trivial.
(3)$\impliedby$(4): If $m\leq k<n$ and (C1$_k$) holds, then $a_{k+1}<r k+1$, thus $a_{k+1}\leq r k$, and (C1$_{k+1}$) follows by using $$\sum_{i>k} (r k)\wedge a_i\leq r k +\sum_{i>k+1} (r(k+1))\wedge a_i.$$
Thus, (C1$_m$)$,\dotsc,$(C1$_n$) hold.
Since $a_1\geq\cdots\geq a_m$ we have $a_i\geq r(m-1)$ for all $i\in\{1,\dotsc,m\}$. Using this and considering the maximal matrix, it is clear that $a^C$ is nonincreasing on $(1,\dotsc,m)$. Apply Lemma \ref{main_lemma} to $(a_1,\dotsc,a_m)$ and $(a^C_1,\dotsc,a^C_m)$.
(5)$\iff$(6): Apply Proposition \ref{conjugate}, and use the fact that when $k<m$ we have $a_k\geq a_m\geq r(m-1)\geq r k$, so the formula $\sum_{i\leq k} a_i^C =\sum_{i\leq r k}(a'_i -1)$ applies.
\end{proof} 

As in the case of bipartite graphs, it is not always necessary to check all the corners less than $m$:
If $k_1<k_2<k_3\leq m$ are three consecutive corners such that $a_{k_1} = a_{k_2}+1 = a_{k_3}+2$ and both (C1$_{k_1}$) and (C1$_{k_3}$) hold, then (C1$_{k_2}$) holds as well.

\subsubsection{Erd\H{o}s-Gallai and Berge criteria for simple undirected graphs}

At long last, we come to the simple undirected graph. Thanks to our efforts on $r$-multigraphs, we can handle it as the special case $r = 1$. A simple undirected graph with degree sequence $a$ can be represented by a symmetric binary matrix with zero diagonal, having column and row sums $a$. Given $a\in\N^n$, we say that $a$ is \emph{graphic} if there exists a simple undirected graph with degree sequence $a$. Given $a\in\N^n$ nonincreasing such that $\sum a_i$ is even, Erd\H{o}s and Gallai \cite{Erdos_1960} proved that $a$ is graphic if and only if
\begin{equation}\tag{EG$_k$}
\sum_{i\leq k} a_i\leq k(k-1)+\sum_{i>k} k\wedge a_i
\end{equation}
for all $k\in\{1,\dotsc,n\}$. With reference to our discussion of $r$-multigraphs, the reader will see that this is simply (C1$_k$) in the case of $r = 1$. As before, this criterion does not lend itself to description in terms of generalized conjugates. The equivalent criterion we gave in the case of $r$-multigraphs specializes to a criterion due to Berge \cite{Berge_1970}. Let $a^E$ be the generalized conjugate given above for $r$-multigraphs, in the case of $r = 1$. It satisfies the identity
$\sum_{i\leq k} a^E_i =\sum_{i\leq k} (k-1)\wedge a_i +\sum_{i>k} k\wedge a_i, $
and in fact
$$a^E_k =\#\{i: 1\leq i<k, \hs a_i\geq k-1\} +\#\{i: k<i\leq n,\hs a_i\geq k\}. $$
When $a$ is nonincreasing, $a^E$ has been referred to as the \emph{corrected conjugate} \cite{Berge_1970}, and in this case, if $a_k\geq k$ then $a^E_k = a'_k-1$.
Consider the inequality:
\begin{equation}\tag{BG$_k$}
\sum_{i\leq k} a_i\leq \sum_{i\leq k} (a'_i-1).
\end{equation}
The following theorem is a direct consequence of Theorem \ref{Chen}.

\begin{theorem}[Reduced Erd\H{o}s-Gallai, Berge] Let $a\in\N^n$ be nonincreasing with $\max a_i\leq n-1$ and $\sum a_i$ even. Let $m =\max\{i: a_i\geq i\}$. The following are equivalent:
\begin{enumerate}
\item $a$ is graphic
\item $a\before a^E$
\item \textup{(EG$_k$)} for all $k\in\{1,\dotsc,n\}$
\item \textup{(EG$_k$)} for $k = m$ and all $k\in C(a)$ such that $k<m$
\item \textup{(BG$_k$)} for $k = m$ and all $k\in C(a)$ such that $k<m$
\item \textup{(BG$_m$)} and $j a'_j+\sum_{i>j} a'_i\leq\sum_{i\leq a'_j} (a'_i-1)$ for all $j \in C(a')$ such that $a'_j<m$.
\end{enumerate}
\end{theorem}

Li \cite{Li_1975} proved the reduction to $k\leq \max\{i: a_i\geq i-1\}$. Eggleton \cite{Eggleton_1975} proved the reduction to $k\in C(a)$. The further reduction to $k = m$ and all $k\in C(a)$ such that $k<m$ is due to Zverovich and Zverovich \cite{Zverovich_1992}, who also observed that some corners can be skipped (in the same way as before).

As in the case of bipartite graphs, let us investigate further reductions in a couple special cases. The first part of the following result is a key lemma in an interesting recent paper by Barrus, Hartke, Jao, and West \cite{Barrus_2011}. The proof below is another nice example of the clarity afforded by using finite calculus.

\begin{proposition} Assume $a\in\N^n$ is nonincreasing, $\sum a_i$ is even, $p =\minimum a_i \geq 1$, and $q =\max a_i \leq n-1$. Denote $\bar a_j =\#\{i: a_i = j\}$. Suppose:
\begin{enumerate}
\item[\textup{(A)}] $\bar a_j\leq 1$ for all $j$ s.t. $p<j<q$, except possibly one $j$ for which $\bar a_j=2$,
\item[or]
\item[\textup{(B)}] $\bar a_j\geq 1$ for all $j$ s.t. $p<j<q$, except possibly one $j$ for which $\bar a_j=0$.
\end{enumerate}
Then $a$ is graphic if and only if \textup{(EG$_m$)} holds, where $m =\max\{i: a_i\geq i\}$.
\end{proposition}
\begin{proof} Let $c_k =\sum_{i\leq k}(a'_i-a_i-1)$ for $k\in\{1,\dotsc,n\}$, noting that $a$ is graphic if and only if $c_1,\dotsc,c_m\geq 0$. Also, (EG$_m$) if and only if $c_m\geq 0$.

(A) Suppose (A) and (EG$_m$) hold. Let $r =\max\{p,\bar a_q\}$. Then $c_m \geq 0$ and
$$0\leq c_1\leq\cdots\leq c_p$$
since $\.c_k = a'_k-a_k-1 =n-a_k -1\geq 0$ for all $k\in\{1,\dotsc,p\}$. Further, 
$$c \mbox{ is concave on } (p,\dotsc,r)$$
since $\: c_k =\.a'_k\leq 0$ for all $k\in\{p+2,\dotsc,r\}$. So if $r\geq m$, then $c_1,\dotsc,c_m\geq 0$. 

Suppose, then, that $r<m$. Let $k$ such that $r<k\leq m$, and let $l = a'_k$. Then $l\geq k$ (since $k\leq m$) and $q>a_k\geq\cdots\geq a_l>p$ (since $k>r\geq\bar a_q$ and $a_l\geq k>r\geq p$). Thus, condition (A) applies to $a_k,\dotsc,a_l$, so $a_k-a_l\geq l-k-1$. Using $l = a'_k$, this yields $\.c_k = a'_k-a_k-1\leq k-a_{a'_k}\leq 0$. Hence, 
$$c_r\geq\cdots\geq c_m\geq 0. $$
Combined with the fact that $0\leq c_1\leq\cdots\leq c_p$ and $c$ is concave on $(p,\dotsc,r)$, this shows that $c_1,\dotsc,c_m\geq 0$.

(B) Suppose (B) and (EG$_m$) hold. Just as in the proof for (A), $c_m \geq 0$ and
$$0\leq c_1\leq\cdots\leq c_p.$$
Now, condition (B) implies that $\.a_k\geq-1$ for all $k\in\{2,\dotsc,n\}$ except possibly one $k_0$ for which $\.a_{k_0} =-2$, and $\.a'_k\leq-1$ for all $k\in\{p+2,\dotsc,q\}$ except possibly one $k_1$ for which $\.a'_{k_1} =0$.
If such an exception occurs, $k_0$ and $k_1$ cannot both be less or equal to $m$ (since if $k_0\leq m$ then $k_1=a_{k_0}+2$ and $a_{k_0}\geq a_m \geq m$ imply that $k_1>m$; similarly, if $k_1\leq m$ then $k_0=a'_{k_1}+1$ and $a'_{k_1}\geq a'_m \geq m$ imply $k_0>m$.)
Hence, $\: c_k =\.a'_k-\.a_k\leq 0$ for all $k\in\{p+2,\dotsc,m\}$ except possibly one $k$ such that $\: c_k = 1$. Thus, 
$$c \mbox{ is almost concave on } (p,\dotsc,m),$$
and we have $c_1,\dotsc,c_m\geq 0$.
\end{proof}





\subsection{Tournaments}
\label{section:tournaments}

A \emph{tournament} is a directed complete graph --- that is, a simple directed graph such that for each pair of vertices $v_1,v_2$ (with $v_1\neq v_2$) exactly one of the edges $(v_1,v_2)$ or $(v_2,v_1)$ appears. A tournament on $n$ vertices can be represented by a matrix $A\in\{0,1\}^{n\times n}$ such that $A_{i j} + A_{j i}= I(i\neq j)$ (where $I(E)$ is $1$ if $E$ is true, and is $0$ otherwise). Given $a\in\N^n$, we say that $a$ is a \emph{score sequence} if $a$ is the degree sequence of a tournament. In a classic result, Landau \cite{Landau_1953} proved that given $a\in\N^n$ non\italic{decreasing} with $\sum a_i ={n\choose 2}$, we have that $a$ is a score sequence if and only if
$$\sum_{i\leq m} a_i\geq{m\choose 2}$$
for all $m\in\{1,\dotsc,n\}$. We can coerce this criterion into our canonical form as follows. Let $S$ be the set of $n\times n$ binary matrices that correspond to tournaments. Then there is a unique maximal matrix in $S$, namely, the matrix $A$ with ones below the diagonal, and zeros elsewhere (that is, $A_{i j} = I(i>j)$). Thus, letting $b^L$ be the corresponding generalized conjugate, we have $b^L_k = n-k$ and $\sum_{i\leq k} b^L_i ={n\choose 2}-{n-k\choose 2}$. (Note that $b^L$ does not depend on the degree sequence in any way.) If $a$ is non\italic{increasing} (and $\sum a_i = {n \choose 2}$), then Landau's criterion can be rewritten as
\begin{equation}\tag{L$_k$}
\sum_{i\leq k} a_i\leq {n\choose 2}-{n-k\choose 2}
\end{equation}
for all $k\in\{1,\dotsc,n\}$, that is, $a\before b^L$. Simply by observing that $b^L$ is nonincreasing, we obtain the corner reduction (which has been previously noted by Beineke \cite{Beineke_1989}). We also get a stronger reduction, in the form of (5) below.

\begin{theorem}[Reduced Landau] \label{Landau} Let $a\in\N^n$ be nonincreasing such that $\sum a_i ={n\choose 2}$. The following are equivalent:
\begin{enumerate}
\item $a$ is a score sequence
\item $a\before b^L$
\item \textup{(L$_k$)} for all $k\in\{1,\dotsc,n\}$
\item \textup{(L$_k$)} for all $k \in C(a)$
\item \textup{(L$_k$)} for all $k\in\{1,\dotsc,n-1\}$ such that $a_k>n-k>a_{k+1}$
\item $j a'_j +\sum_{i>j}a'_i\leq{n\choose 2}-{n-a'_j\choose 2}$ for all $j\in C(a')$.
\end{enumerate}
\end{theorem}
\begin{proof} 
(1)$\iff$(3): Landau. (2)$\iff$(3): Trivial. (3)$\implies$(4): Trivial. (4)$\implies$(5): Trivial. (3)$\impliedby$(5): For $k\in\{1,\dotsc,n\}$, let $c_k = {n\choose 2}-{n-k\choose 2}-\sum_{i\leq k} a_i$ and observe that $\.c_k =(n-k)-a_k$ for all $k\in\{1,\dotsc,n\}$. To ensure that $c_k\geq 0$ for all $k$, it is sufficient to check it at the ``local minima'' --- more precisely, it is sufficient that $c_k\geq 0$ for all $k\in\{1,\dotsc,n-1\}$ such that $\.c_k<0$ and $\.c_{k+1}\geq 0$, or equivalently, $a_k>n-k>a_{k+1}$.
(4)$\iff$(6): Proposition \ref{conjugate}.
\end{proof}

\begin{corollary} Let $a\in\N^n$ such that $\sum a_i ={n\choose 2}$. Then $a$ is a score sequence if $\{a_1,\dotsc,a_n\}$ contains every integer in $\{\min a_i,\dotsc,\max a_i\}$ except possibly one of them.
\end{corollary}
\begin{proof} Rearrange $a$ to be nonincreasing. Consider the sequence $c$ defined in the proof of Theorem \ref{Landau}. By our assumptions, $\:c_k\leq 0$ for all $k\in\{2,\dotsc,n\}$, or there exists $l>1$ such that $\:c_l\leq 1$ and $\:c_k\leq 0$ for all $k>1$, $k\neq l$.  Apply Corollary \ref{concave_corollary}.
\end{proof}

\subsection{A negative result}
\label{negative_example}

In light of the generality with which the preceding reductions apply, it is natural to wonder if there is some more fundamental principle underlying all these results.  A common thread running through all these examples is that the degree sequence criterion takes the form $a\before b^*$ for some generalized conjugate $b^*$. Chen \cite{Chen_1992} has proven a very general criterion of this form, applicable to many classes of graphs. This raises the question of whether it is possible to extend these results to every class satisfying Chen's conditions, or even more broadly: given a class of graphs having a degree sequence criterion of the form $a\before b^*$ (for a generalized conjugate $b^*$), does the corner reduction always apply? (In other words, is it always sufficient to check the inequalities at the indices $k$ such that $a_k>a_{k+1}$ when $a$ is nonincreasing?) In this subsection, we answer this question in the negative by exhibiting a simple counterexample.

Let
$$
C = \left(\begin{matrix}  
                    0 & 0 & 1 & 1 \\
                    0 & 0 & 1 & 1 \\
                    1 & 1 & 1 & 1 \\
                    1 & 1 & 1 & 1 \end{matrix}\right).
$$
Given $b\in\N^4$ such that $b\leq(2,2,4,4)$, let $S\subset\{0,1\}^{4\times 4}$ be the subset of $4\times 4$ binary matrices $A$ having row sums $b$ and satisfying $A\leq C$ (that is, the upper left $2\times 2$ block is forced to be zero.) Denote by $b^C$ the corresponding generalized conjugate. The set $S$ corresponds to a very simple class of bipartite graphs, and it can be shown that if $a\in\N^4$ is nonincreasing with $\sum a_i =\sum b_i$, then there exists a binary matrix $A\leq C$ with column sums $a$ and row sums $b$ if and only if $a\before b^C$. (Perhaps the easiest way to show this is by an argument similar to Krause's proof \cite{Krause_1996} of the Gale-Ryser criterion.) However, it is insufficient that the inequalities hold only at the indices $k$ such that $a_k>a_{k +1}$, as the following example illustrates. Choosing $b =(2,1,1,1)$, the maximal matrix is
$$ A^*= \left(\begin{matrix}
                    0 & 0 & 1 & 1 \\
                    0 & 0 & 1 & 0 \\
                    1 & 0 & 0 & 0 \\
                    1 & 0 & 0 & 0 \end{matrix}\right),
$$
so $b^C =(2,0,2,1)$. Now, choose $a =(2,1,1,1)$ also, noting that $a$ is nonincreasing and $\sum a_i =\sum b_i$. Clearly there is no matrix $A\leq C$ with these row and column sums, and as expected, $a \ntrianglelefteq b^C$. However, the inequalities hold at the corners of $a$ (since $a_1\leq b^C_1$ and $\sum_{i = 1}^4 a_i\leq \sum_{i = 1}^4 b^C_i$). Hence, the corner reduction does not apply. 

This shows that the corner reduction does not always apply for a class of graphs with a degree sequence criterion of the form $a\before b^*$. One might hope that it would still hold for those classes satisfying Chen's conditions, but the same example shows that this is not the case.
Rather than give a full exposition of Chen's criterion, we simply remark that this example satisfies his conditions (that is, it is what he refers to as a monotone class satisfying his ``main condition'').

\section{Concluding remarks}

We have proven reduced degree sequence criteria for several diverse classes of graphs. By introducing the use of finite calculus, coupled with generalized conjugates, we have presented a unified approach to these problems. We have applied this framework to obtain many new results, and to offer more interpretable proofs for those results which were previously known.

\section*{Acknowledgments}
The author was supported by a NDSEG fellowship.


\bibliography{references}
\bibliographystyle{model1b-num-names}

\end{document}